\documentclass{article}
\usepackage[cp1251]{inputenc}
\usepackage{amssymb, amsthm, amsmath}
\usepackage{graphicx}
\usepackage{float}
\usepackage[usenames]{color}
\usepackage{colortbl}
\usepackage{subfigure}
\newtheorem{theorem}{Theorem}
\newtheorem{ex}{Example}
\newtheorem{statement}{Statement}

\begin{document} 
\date{}
\title{\textbf{Infinitely many graph manifolds with unique geometrical piece that admit arbitrarily many Anosov flows}}
\author{O. Pochinka, V. Shmukler}
\maketitle

\begin{abstract} Anosov flows  have a long and rich history, firstly motivated by the 
study of geodesic flows in negative curvature surface by Anosov and Sinai. Not every closed manifold admits an Anosov flow for well-known reasons: the fundamental group of a 3-manifold \(M\) admitting an Anosov flow must have exponential growth, and \(M\) must be universally covered by \(\mathbb{R}^{3}\). Nevertheless, there are sufficient mechanisms for constructing distinct Anosov flows on admissible 3-manifolds, such as Dehn-Goodman-Fried surgery or playing with hyperbolic building blocks. A central problem in the field has been to determine the number of Anosov flows that can be supported by a single manifold. The question of whether there exists an infinite set of pairwise non-equivalent 
Anosov flows on a 3-manifold remains open to this day. However, there are several 
papers proving the existence of a manifold $M_n$ that admits $n$ pairwise inequivalent 
Anosov flows for any natural number $n$. In all known examples, the manifolds $M_n$ 
are composed of several geometric pieces. In the present paper, we prove the 
existence of a countable number of graph manifolds $M_{k,n}$, $k \in \mathbb{N}$ 
with a single geometric piece, each of which admits $n$ pairwise non-equivalent 
transitive Anosov flows. All previously known constructions of different flows 
on the same graph manifold were based on gluing geodesic flows. The nature of 
the flows constructed in this paper is completely different; they are constructed 
from a single hyperbolic plug, which is a suspension over a Morse-Smale 
diffeomorphism on a surface. 
\end{abstract}

\section{Introduction}
Let $M$ be a closed smooth manifold. We say that a smooth vector field $X$ in $M$ is {\it an Anosov vector field} if denoting $\phi^t$ the flow generated by $X$ ({\it Anosov flow}) it holds that there is a $D\phi^t$-invariant continuous splitting $$TM = E^s\oplus E^c\oplus E^u,$$ where \(E^{c}\) is the 1-dimensional neutral direction tangent to the flow and there is $T > 0$ so that for vectors $v^s\in E^s$ and $v^u\in E^u$ one has that
\begin{equation}\label{hyp}
D\phi^T(v^s)||\leqslant\frac12\,||v^s||,\,\ ||D\phi^T(v^u)||\geqslant 2\,||v^u||.
\end{equation}

Among many interesting properties (see, for example, \cite{KH}), Anosov vector fields are structural stable, Anosov flows always have infinitely
many periodic trajectories, all of them saddle type hyperbolic. The Stable Manifold Theorem asserts that the plane bundles $E^s\oplus E^c$ and $E^u\oplus E^c$
are tangent to a pair of transverse
codimension one foliations $F^s$ and $F^u$ on $M$, that intersects along the $\phi^t$-orbits. We call them {\it stable and  unstable foliations}. Every leaf of these foliations is homeomorphic to either a cylinder or a Mobius band (in case the leaf contains a periodic orbit), or to a plane. Also there are {\it strong stable and unstable} foliations $F^{ss}$ and $F^{uu}$ tangent to $E^s$ and $E^u$, accordingly.

Two important examples, which were those which motivated the definition are the
following. 

\begin{ex}[Suspensions] Consider a linear map $A\in SL(d,\mathbb Z)$ so that all its eigenvalues have modulus different from one. It induces a diffeomorphism $f_A :\mathbb T^d\to\mathbb T^d$ where we
identify $\mathbb T^d\cong\mathbb R^d/\mathbb Z^d$ ({\it Anosov diffeomorphism}). Consider the manifold $M_A = \mathbb T^d\times\mathbb R/_\sim$, where $(x, s)\sim (f^k_A(x), s-k)$
for every $x\in\mathbb T^d,\, t\in\mathbb R,\, k\in\mathbb Z$. The flow $\phi^t: M_A\to M_A$ given by $\phi^t([(x, s)]) = [(x, s+t)]$ is Anosov. 
\end{ex}

\begin{ex}[Geodesic flows] Let $(S, g)$ be a smooth closed 2-dimensional Riemannian manifold with constant negative curvature.
The phase space of the dynamics is the unit tangent bundle $T^1S$, defined as 
$$T^1S = \{ (x, v) \in TS \mid g_x(v, v) = 1 \}.$$ 
For any initial state $(x, v)$ in 3-manifold $T^1S$, the theorem of existence and uniqueness for ordinary differential equations ensures a unique, globally defined, constant-speed geodesic $\gamma=\gamma_{(x,v)}: \mathbb{R} \to S$ such that
$$\gamma(0) = x \quad \text{and} \quad \dot{\gamma}(0) = v.$$ 
The geodesic flow is a one-parameter group of diffeomorphisms $\phi^t: T^1S \to T^1S$ mapping each vector forward along its corresponding geodesic
$$\phi^t(x,v) = \left( \gamma(t), \, \dot{\gamma}(t) \right) \in T^1S.$$ 
The base projection $\pi: T^1S \to S$ maps the orbit $\phi^t(x,v)$ back onto the  curve $\gamma(t)$. The vector $\dot{\gamma}(t)$ is obtained via parallel transport along the curve $\gamma(t)$.
\end{ex}

Not every closed manifold admits an Anosov flow. So, Margulis \cite{Marg} proved the following beautiful result which provided restrictions on the topology of a manifold that carries an Anosov flow.

\begin{statement}[\cite{Marg}]\label{exp} A 3-manifold admitting an Anosov
flow has the fundamental group with  exponential growth.
\end{statement}

For instance, the 3-dimensional sphere, lens spaces, the three dimensional torus, $\mathbb S^2\times\mathbb S^1$ or nilmanifolds cannot admit Anosov flows. 

The following result excludes other classes of manifolds, such as connected sums.

\begin{statement}[\cite{PT}, \cite{CC}, \cite{Ca}]\label{con} A 3-manifold admitting an Anosov flow must be irreducible and moreover, its universal cover is homeomorphic to $\mathbb R^3$.
\end{statement} 

Ghys \cite{Gh} provides a strong classification of Anosov
flows in circle bundles (the completion of this classification is contained in the recent \cite{BaFe3}) and states the following.

\begin{statement}[\cite{Gh}, \cite{BaFe3}]\label{T1S}
If a circle bundle over a compact surface $S$ admits an Anosov flow then  it is a finite cover of the unit tangent bundle of $S$.
\end{statement}

\begin{statement}[\cite{Barb2}]\label{Sei} Any Anosov flow on a closed Seifert fibered space is topologically equivalent to a finite lift of a geodesic flow on a hyperbolic surface.
\end{statement}

For a toroidal 3-manifold $M$, the collection of  JSJ-decomposition torus is in invariant of Anosov flow on $M$ in a sense of the following result.  

\begin{statement}[\cite{BaJ}, \cite{Br}]\label{br} Let $Z$ be an Anosov vector field on a closed three-manifold
$M$ and let $T$ and $T'$ be some tori embedded in $M$ and transverse to $Z$. If $T$ is homotopic to $T'$, then $T$ is isotopic to $T'$ along the orbits of $Z$.
\end{statement}

A central problem in the field has been to determine the number of Anosov flows
that can be supported by a single manifold.

Works of Plante, Ghys and Barbot \cite{Pla81}, \cite{Gh}, \cite{Barb2}  give finiteness results on certain geometric (solvable or Seifert fibered) 3-manifolds. By contrast, much more recent work shows that, for any given $n$, one
can find examples of closed 3-manifolds (either hyperbolic or with non-trivial JSJ-decomposition) that admit at least $n$  inequivalent transitive Anosov flows.

The first explicit examples of manifolds supporting multiple Anosov flows were
constructed in \cite{Barb2}, where Barbot constructs a family of graph manifolds that each
support two distinct Anosov flows. The surgery techniques of Goodman \cite{Goo} may
also produce two distinct Anosov flows on a manifold if the periodic orbit used for
the Dehn surgery admits two distinct purely cosmetic surgery slopes. More recently,
Bonatti, Beguin and Yu proved a general theorem allowing them to glue many pieces
with transverse toral boundaries \cite{BBY}. They use this new technique to construct, for each $n\in\mathbb N$, a non-geometric three-manifold $M_n$ consisting of two hyperbolic pieces and one Seifert fibered piece which supports at least $n$  distinct Anosov flows. Similar result was recently obtained by Bowden and Mann using an analysis of the rigidity properties of certain fundamental group actions [7], but their manifolds $M_n$ are hyperbolic. For each $n\in\mathbb N$ Clay and Pinsky \cite{ClPi} constructed a graph manifold which supports at least $n$ different Anosov flows. Their manifolds have only two pieces, each piece being the exterior of a trefoil with a geodesic flow. 

In the present paper, we prove the following result.

\begin{theorem}\label{tigr} For every $n\in\mathbb N$ there are  countable many graph manifolds $M_{k,n}$, $k \in \mathbb{N}$ 
with a single geometric piece, each of which admits $n$ pairwise inequivalent 
transitive Anosov flows.
\end{theorem}

All previously known constructions of different flows on the same graph manifold were based on gluing geodesic flows. The nature of the flows constructed in this paper is completely different; they are created from a single hyperbolic plug, which is a suspension over a Morse-Smale 
diffeomorphism on a surface. 

{\it Acknowledgement.} This work is an output of a research project (HSE-BR-2025-84) implemented as part of the Basic Research Program at HSE University.

\section{Hyperbolic plugs}
In this section, we briefly outline the main ideas of the Beguin, Bonatti, and  Yu paper \cite{BBY}, illustrating the text with elementary examples. 

A {\it plug} is a pair $(V,X)$ where $V$ is a compact 3–manifold with boundary and $X$ is a vector field on $V$, transverse to  $\partial V$ (in particular, $X$ is assumed to be non-singular on $V$). Given such a plug $(V,X)$, we can decompose $\partial V$ as the disjoint union of an entrance boundary $T^{+}_X$ (the part of $\partial V$, where $X$ is pointing
inwards) and an exit boundary $T^{-}_X$ (the part of $\partial V$ where $X$ is pointing outwards).
The plug $(V,X)$ is called an {\it attracting plug} if $T^{-}_X=\varnothing$,  a {\it repelling plug}  if $T^{+}_X=\varnothing$ and {\it saddle plug}, otherwise. The plug $(V,X)$ is called a {\it hyperbolic plug} if  the maximal invariant set $\Lambda=\bigcap\limits_{t\in\mathbb R} X^t(U)$ of generated $X$ flow $X^t$,  forms a hyperbolic set with one-dimensional strong stable and unstable bundles. 

If $(V,X)$ is a hyperbolic plug, the stable lamination $W^s_\Lambda$ (resp. the unstable lamination $W^u_\Lambda$)  intersects transversally $T^+_X$ (resp. $T^-_X$) and is disjoint from $T^-_X$ (resp. $T^+_X$). By transversality, $\mathcal L^s_X=W^s_\Lambda\cap T^+_X$ and $\mathcal L^u_X=W^u_\Lambda\cap T^-_X$
are one-dimensional laminations. 

The laminations $\mathcal L^s_
X$ and $\mathcal L^u_X$ are {\it  Morse-Smale laminations (MS-laminations)} \cite[Proposition 3.8]{BBY}, that is:
\begin{itemize}
\item(i) they contain finitely many compact leaves;
\item(ii) every half non-compact leaf is asymptotic to a compact leaf;
\item(iii) each compact leaf may be oriented such that its holonomy is a contraction, (this orientation is called {\it the contracting orientation}).
\end{itemize} 

\begin{ex}\label{net} For example, let as consider a gradient-like diffeomorphism $F_{1,0} : \mathbb{T}^2 \to \mathbb{T}^2$ (its phase 
portrait is shown in Figure \ref{0}).
\begin{figure}[htbp]
	\centering
	\includegraphics[width=\linewidth]{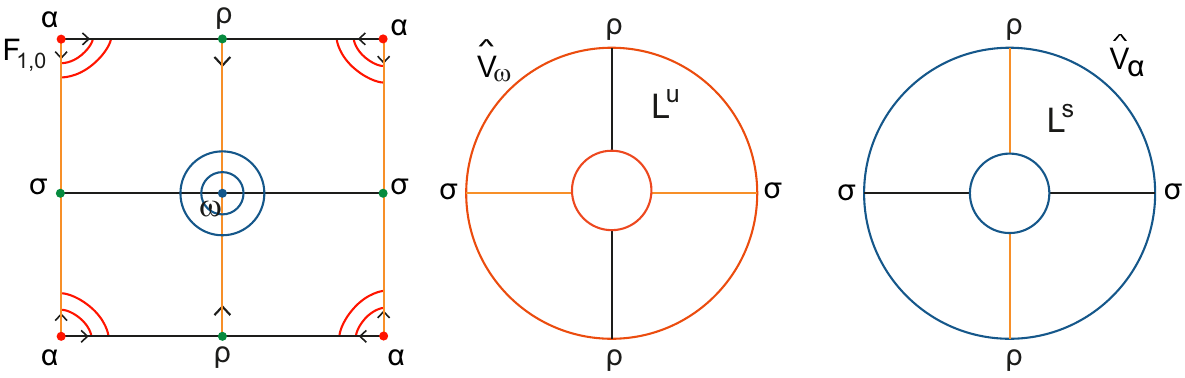}
	\caption{Diffeomorphisms $F_{1,0}$}
	\label{0}
\end{figure}
The non-wandering set of the diffeomorphism $F_{1,0}$ consists of one 
sink $\omega$, one source $\alpha$, and two saddles $\sigma,\rho$. The orbit space $\hat V_{\alpha}$ of the source basin $V_{\alpha}=W^u_{\alpha}\setminus\alpha$ 
is homeomorphic to a two-dimensional torus, and the stable saddle 
separatrices $L^s$ form a lamination in it consisting of four pairwise 
disjoint knots. The lamination $L^u$, which is the projection of the 
unstable saddle separatrices into the orbit space $\hat V_{\omega}$ 
of the sink basin $V_{\omega}=W^s_{\omega}\setminus\omega$, looks completely analogous. The hyperbolic plug $(V,X)$ is obtained from the suspension over the 
diffeomorphism $F_{1,0}$ by removing the neighborhoods of the sink 
and source periodic orbits. The plug is a saddle and its set $\Lambda$ consists of two saddle periodic orbits corresponding the saddles $\sigma,\rho$. Moreover, there are homeomorphisms $h_+:\hat V_{\alpha}\to T^+_X,\,h_-:\hat V_{\omega}\to T^-_X$ such that $h_+(L^s)=\mathcal L^s_X,\,h_-(L^u)=\mathcal L^u_X$. Thus, $\mathcal L^s_
X,\,\mathcal L^u_X$ are MS-laminations without non-compact leaves.  
\end{ex} 

\begin{statement}[\cite{BBY}, Proposition 1.1]\label{glu} Let $(V,X)$ and $(U,Y)$ be two hyperbolic plugs. Let $\tilde T^-_X$ be a union of connected components of $T^-_X$ and $\tilde T^+_Y$ be a union of connected components of $T^+_Y$. Assume that there exists a diffeomorphism $\varphi:\tilde T^-_X\to\tilde T^+_Y$ such that $\varphi(\mathcal L^u_X)$ is transverse to $\mathcal L^s_Y$. Let $Z$ be the vector field induced by $X$ and $Y$ on the
manifold $W=V\cup_\varphi U$. Then $(W,Z)$  is a hyperbolic plug.
\end{statement}

If $(V,X)$ is a hyperbolic plug such that $V$ is embedded in a closed three-dimensional
manifold $M$ and $X$ is the restriction of an Anosov vector field $\bar X$ on $M$, then the
stable (resp. unstable) lamination of the maximal invariant set of $V$ is embedded in
the stable (resp. unstable) foliation of the Anosov vector field $\bar X$ . This leads to some restrictions on the entrance and exit laminations of $V$, and motivates the following definition.

A lamination $\mathcal L$ on a compact surface $S$ is {\it a filling MS-lamination} if it satisfies
properties $(i), (ii), (iii)$ above and if the accessible boundary of every connected component $C$ of $S\setminus\mathcal L$  consists of two distinct noncompact leaves $L_1,L_2$ which are asymptotic to each other at both ends. Since any filling MS-lamination can be embedded in a $C^{0,1}$-foliation, then $S$ is either a torus or a Klein bottle. 

A hyperbolic plug $(V,X)$ has a filling MS-lamination $\mathcal L^s_X$ if
and only if this is also the case for the $\mathcal L^u_X$ \cite[Lemma 3.21]{BBY}.

The laminations of the plug from Example \ref{net} are not filling. To get a plug with the filling laminations let us consider the following example.

\begin{ex}\label{est} Let $k\in\mathbb N$ and let $F_{1,0,k}$ be the connected sum of two copies of diffeomorphism $F_{1,0}$  obtained by identifying $V_{\omega}$ of the first copy with  $V_{\alpha}$ of the second copy using the $k$-th Dehn twist $D_k:\hat V_{\omega}\to\hat V_{\alpha}$ on the corresponding orbit spaces 
(see Figure \ref{00}).
\begin{figure}[htbp]
	\centering
	\includegraphics[width=0.8\linewidth]{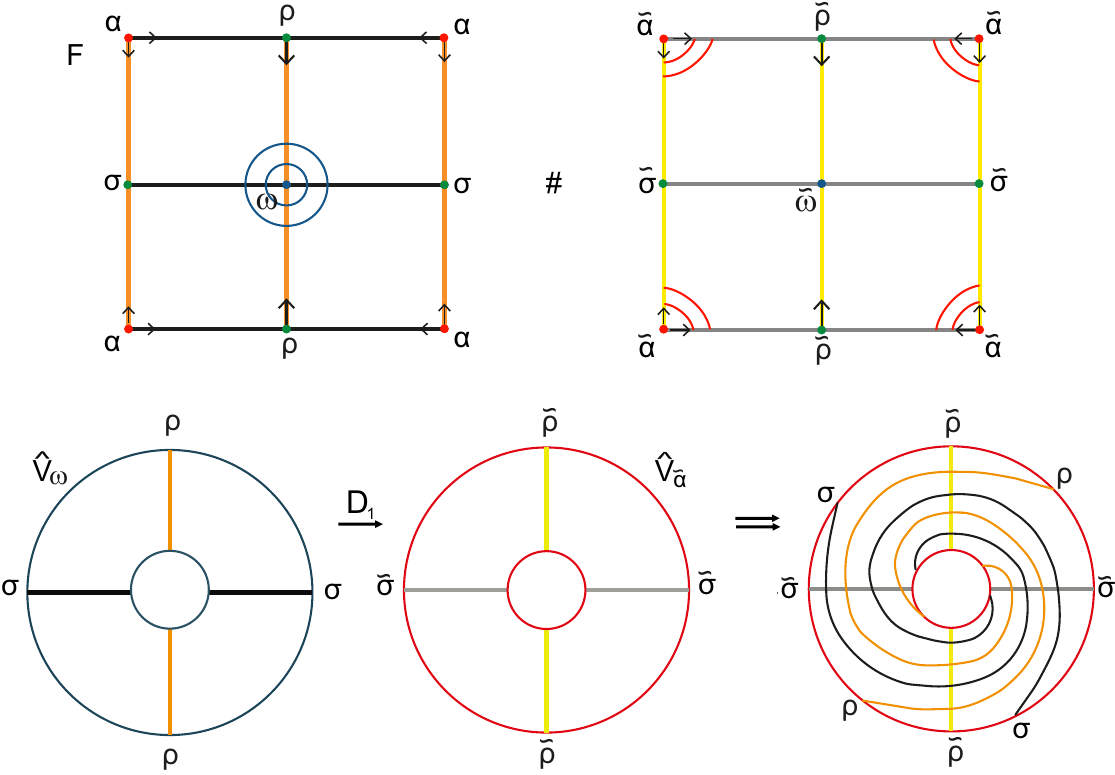}
	\caption{Diffeomorphisms $F_{1,0}\#F_{1,0}$}
	\label{00}
\end{figure} 
Thus, the diffeomorphism $F_{1,0,k}$ is a Morse-Smale diffeomorphism on an 
orientable surface of genus $2$ (see Figure \ref{S2}).
Its non-wandering set consists of one  
sink $\omega$, one sources $\alpha$, and fore  saddles, whose invariant manifolds intersect along $4k$ heteroclinic orbits. The projection of stable $L^s$ and unstable $L^u$ saddle separatrices to $\hat V_{\alpha}$ and  $\hat V_{\omega}$ accordingly is represented on Figure \ref{S2}.
\begin{figure}[htbp]
	\centering
	\includegraphics[width=0.8\linewidth]{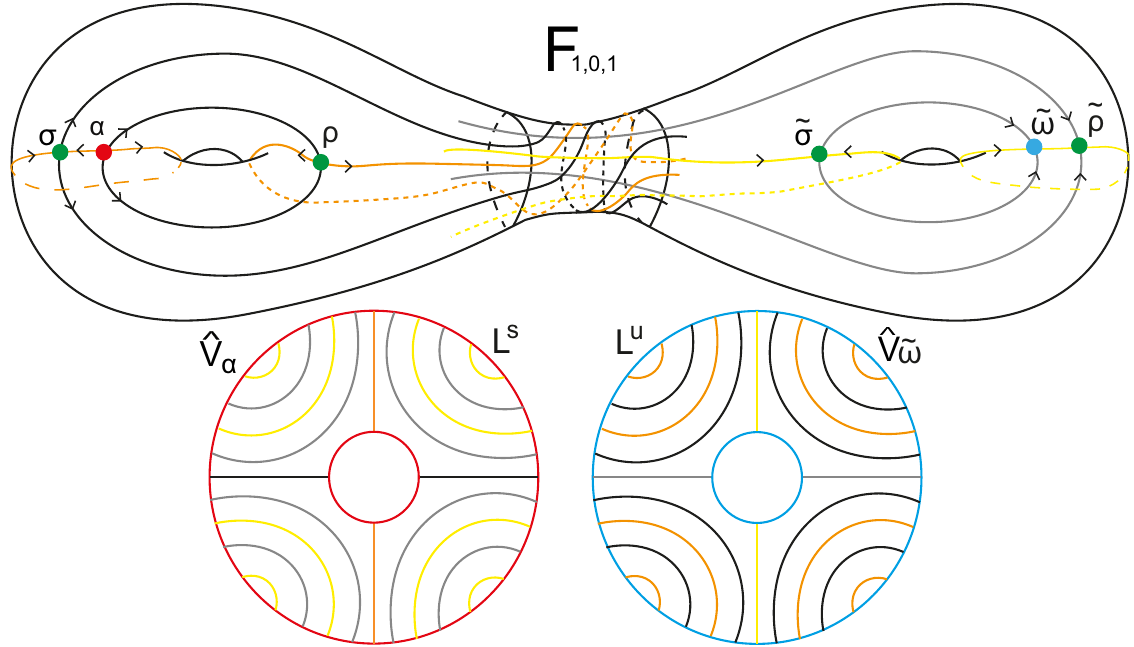}
	\caption{Diffeomorphism $F_{1,0,1}$}
	\label{S2}
\end{figure}
 The hyperbolic plug $(V,X)$ is obtained from the suspension over the 
diffeomorphism $F_{1,0,1}$ by removing the neighborhoods of the sink 
and source periodic orbits. The plug is a saddle and its set $\Lambda$ is a one  basic set consisting of fore saddle periodic orbits and $4k$ heteroclinic orbits. Moreover, there are homeomorphisms $h_+:\hat V_{\alpha}\to T^+_X,\,h_-:\hat V_{\omega}\to T^-_X$ such that $h_+(L^s)=\mathcal L^s_X,\,h_-(L^u)=\mathcal L^u_X$. Thus, $\mathcal L^s_
X,\,\mathcal L^u_X$ are filling MS-laminations.
\end{ex}

Two filling MS-laminations $\mathcal L_1$ and $\mathcal L_2$ on a surface $S$ are called {\it strongly transverse} if they are transverse and if every connected component $C$ of $S\setminus(\mathcal L_1\cup\mathcal L_2)$ is a topological disc whose boundary consists of exactly four segments, which alternatively   
lie on leaves of $\mathcal L_1$ and  $\mathcal L_2$ (see Fig. \ref{str}).
\begin{figure}[htbp]
	\centering
	\includegraphics[width=0.7\linewidth]{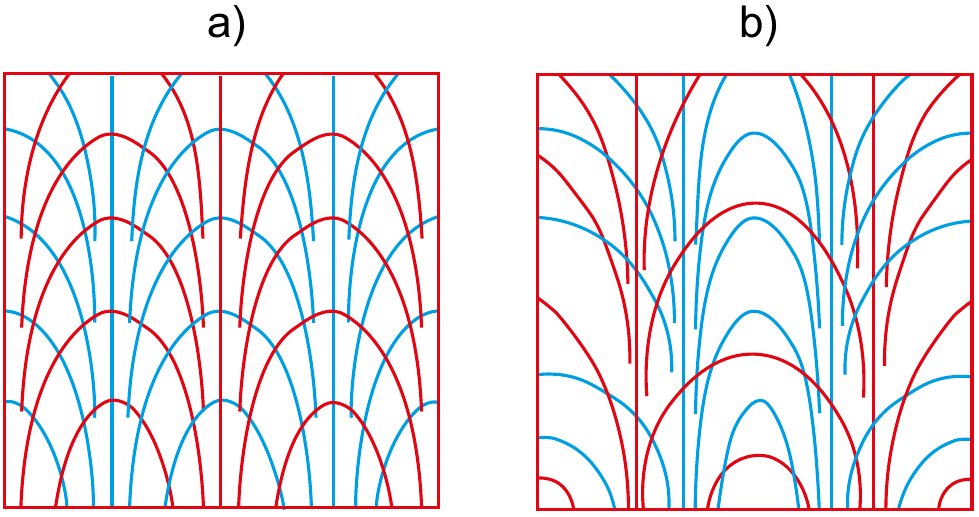}
	\caption{a) strongly transverse; b) transverse, but not strongly transverse filling MS-laminations}
	\label{str}
\end{figure}

\begin{statement}[\cite{BBY}, Proposition 1.3]\label{W} In Proposition \ref{glu}, assume furthermore that the plugs $(V,X)$ and $(U,Y)$  have filling MS-laminations and that $\varphi(\mathcal L^u_
X)$ is strongly transverse to $\mathcal L^s_Y$. Then the plug $(W,Z)$ has filling MS-laminations.
\end{statement}

Let $(V,X_i)$, $i\in\{0,1\}$ be hyperbolic plugs with filling MS-laminations and  strongly transverse gluing diffeomorphisms $\varphi_i: T^-_{X_i}\to T^+_{X_i}$. Triples $(V,X_0,\varphi_0)$, $(V,X_1,\varphi_1)$ is called {\it strongly isotopic} if there is a continuous path $(V,X_t,\varphi_t),\,t\in[0,1]$ of hyperbolic plugs with filling MS-laminations and strongly transverse gluing diffeomorphisms. Notice that this implies that $(V,X_0)$, $(V,X_1)$ are topologically equivalent.

\begin{statement}[\cite{BBY}, Theorem 1.5]\label{mai} Let $(V,X)$ be a hyperbolic plug with filling MS-laminations such
that the maximal invariant set of $X$ contains neither attractors nor repellers, and let $\varphi: T^-_{X}\to T^+_{X}$ be a strongly transverse gluing diffeomorphism. Then there exist a
hyperbolic plug $(V,Y)$ with filling MS-laminations and a strongly transverse gluing diffeomorphism $\psi: T^-_{Y}\to T^+_{Y}$ such that $(V,X,\varphi)$  and $(V,Y,\psi)$ are strongly
isotopic, and such that the vector field $Z$ induced by $Y$ on $V/\psi$ is Anosov.
\end{statement}

To decide whether the Anosov vector field $Z$, generated a triple $(V,X,\varphi)$ in Statement \ref{mai} is transitive or not, let us consider the oriented graph $\Upsilon$ defined as follows:
\begin{itemize}
\item the vertices of $\Upsilon$ are the basic pieces $\Lambda_1,\dots,\Lambda_k$ of $X$;
\item  there is an edge going from $\Lambda_i$ to $\Lambda_j$ if $W^u_{\Lambda_i}$ intersects $W^s_{\Lambda_j}$ or $\varphi(W^u_{\Lambda_i})$ intersects $W^s_{\Lambda_j}$.
\end{itemize} 
We say that $(V,X,\varphi)$ is {\it combinatorially transitive} if any
two edges of $\Upsilon$ can be joined by an oriented path.

\begin{statement}[\cite{BBY}, Proposition 1.6]\label{tra} Under the hypotheses of Statement \ref{mai}, if $(V,X,\varphi)$ is combinatorially
transitive, then the Anosov vector field $Z$ is transitive.
\end{statement}

\section{Anosov flow for the hyperbolic plug $(V,X)$ of Example \ref{est}}
Let us illustrate an idea of Statement \ref{mai} proof for the hyperbolic plug $(V,X)$ of Example \ref{est}. 

As $F_{1,0,1}$ is a Morse-Smale diffeomorphism on a surface, then it has a compatible system of neighborhoods (see, for example, \cite{Pa}, \cite{PS}). 
We present a modification of this notion  for arbitrary MS-diffeomorphism $f$ on a closed surface following to \cite[Theorem 2]{BGLP}. 

For $\nu\in\{-1,+1\}$, let $a_{\nu}:\mathbb R^2\to\mathbb R^2$ denote the diffeomorphism defined by the formula   $$a_{\nu}(x,y)=\left(\nu\cdot\frac{x}{2},\nu\cdot 2 y\right).$$ The diffeomorphism $a_{\nu}:\mathbb R^2\to\mathbb R^2$ has a unique fixed saddle point at the origin $O$ with the stable manifold $W^s_O=Ox_1$ and the unstable manifold $W^u_O=Ox_{2}$. 

Let $\mathcal N=\{(x_1,x_2)\in\mathbb{R}^2~:~ |x_1x_2|< 1\}$. We define a pair of transversal foliations $\mathcal{F}^u,~\mathcal{F}^s$ in the neighborhood $\mathcal N$ as follows:   
$$\mathcal{F}^u=\bigcup\limits_{c_{2}\in Ox_2}\{(x_1,x_2)\in \mathcal N~:~x_{2}=c_{2}\},$$ 
$$\mathcal{F}^s=\bigcup\limits_{c_1\in Ox_1}\{(x_1,x_2)\in \mathcal N~:~x_1=c_1\}.$$

Notice that the set $\mathcal N$ is invariant under the canonical diffeomorphism $a_{\nu}$, and $a_{\nu}$ maps the leaves of the foliation $\mathcal{F}^u$ ($\mathcal{F}^s$) into the leaves of the same foliation.

Let $\sigma$ be a saddle periodic point of the diffeomorphism $f$. Denote by $m_\sigma$ the period of the saddle point $\sigma$. Recall that the {\it orientation type} of the saddle point $\sigma$ is the number $\nu_\sigma$ equal to $-1\,(+1)$ if the diffeomorphism $f^{m_\sigma}|_{W^u_\sigma}$ reverses (preserves) orientation.

A neighborhood $N_\sigma$ of the saddle point $\sigma$ is called {\it linearizing} if there exists a homeomorphism $h_\sigma:N_\sigma\to {\mathcal N}$ conjugating the diffeomorphism $f^{m_{\sigma}}\vert_{{N}_f}$ with the canonical diffeomorphism $a_{\nu_\sigma}|_{\mathcal N}$.  

The foliations $\mathcal{F}^u,\,\mathcal{F}^s$ induce, via the homeomorphism $h_\sigma^{-1}$, $f^{m_{\sigma}}$-invariant foliations ${F}^u_\sigma,\,{F}^s_\sigma$ on the linearizing neighborhood $N_\sigma$ (see Fig. \ref{ris:linear_0}). For any point $x\in N_\sigma$, let $F^u_{\sigma,x},\,F^s_{\sigma,x}$ denote the leaves of the foliations ${F}^u_\sigma,\,{F}^s_\sigma$ passing through the point $x$.
\begin{figure}[htbp]
\center{\includegraphics[width=0.8\linewidth]{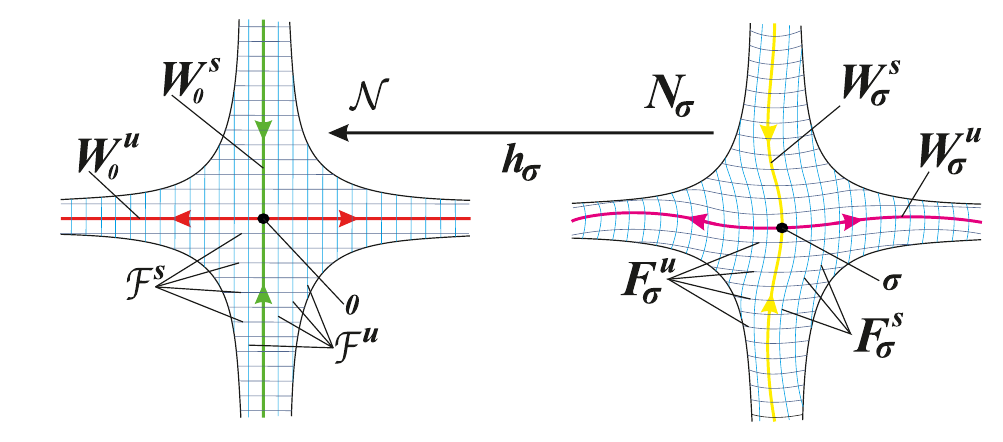}}\caption{Linearizing neighborhood of a saddle point $\sigma$}
\label{ris:linear_0}
\end{figure}

The neighborhood $N_{\mathcal O_\sigma}=\bigcup\limits_{k=0}^{m_{\sigma}-1}f^k(N_\sigma)$, equipped with the map $h_{\mathcal O_\sigma}$ composed of homeomorphisms $h_\sigma f^{-k}:f^k(N_\sigma)\to \mathcal N,~k=0,\dots,m_{\sigma}-1$, will be called the {\it linearizing neighborhood of the orbit} $\mathcal O_\sigma$. 

A set ${N}_f$ of $f$-invariant linearizing neighborhoods ${N}_\sigma$ of all saddle points $\sigma$ of the diffeomorphism $f$ is called a {\it compatible system of neighborhoods} if the following properties hold (see Fig. \ref{ris:neib2}):
\begin{figure}[htbp]
\center{\includegraphics[width=0.8\linewidth]{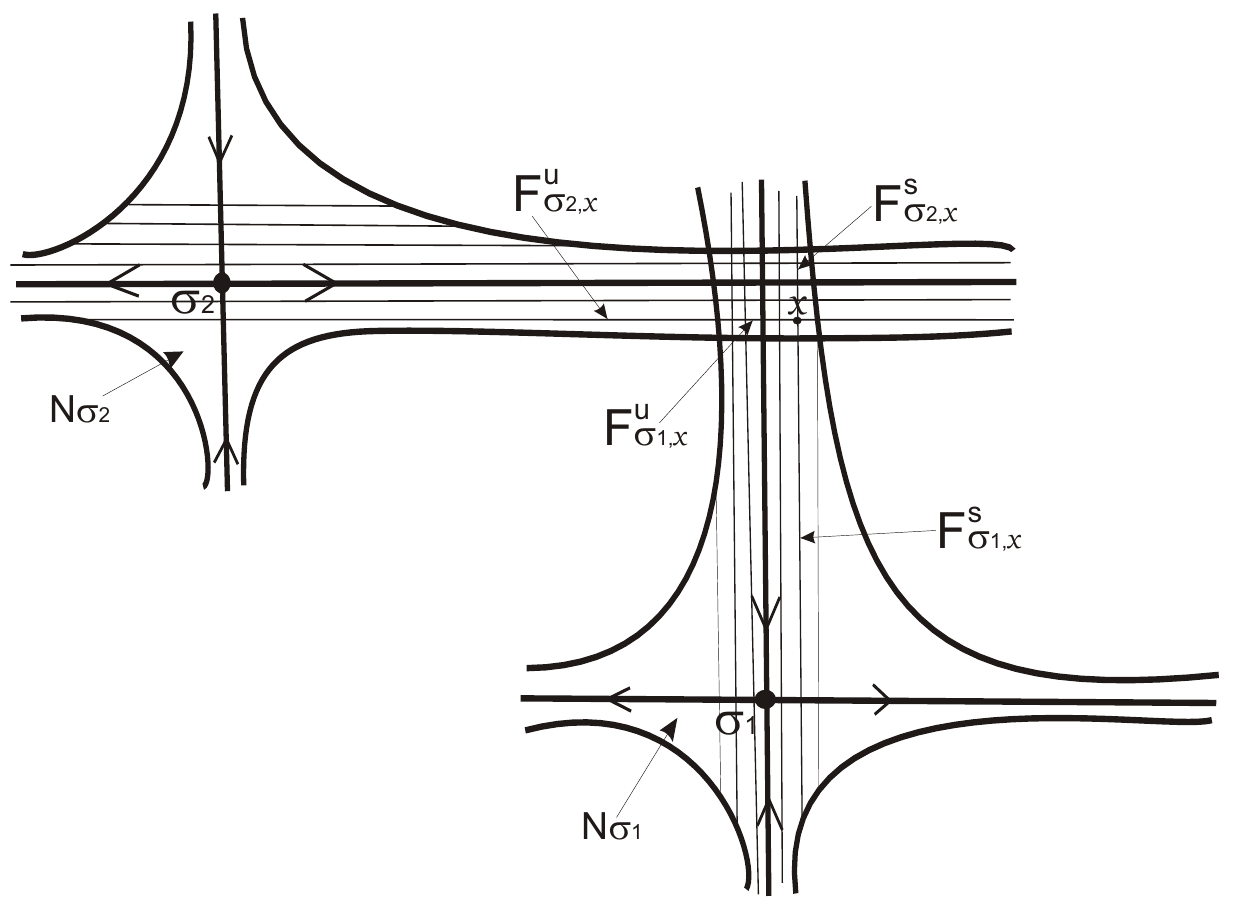}}
\caption{Compatible foliations}
\label{ris:neib2}
\end{figure}
\begin{itemize}
\item If $W^{s}_{{\sigma_1}}\cap W^{u}_{{\sigma_2}}=\varnothing$ and $W^{u}_{{\sigma_1}}\cap W^{s}_{{\sigma_2}}=\varnothing$ for saddle points $\sigma_1,\sigma_2$, then ${{N}}_{{\sigma_1}}\cap{{N}}_{{\sigma_2}}=\varnothing$;
\item If $W^s_{\sigma_1}\cap W^u_{\sigma_2}\neq\varnothing$ for saddle points $\sigma_1,\sigma_2$, then $$(F^u_{\sigma_1,x}\cap{N}_{\sigma_{2}})\subset F^u_{\sigma_2,x},\,\,\,(F^s_{\sigma_2,x}\cap N_{\sigma_{1}})\subset F^s_{\sigma_1,x},$$
for $x\in({N}_{{\sigma_1}}\cap {N}_{{\sigma_2}})$.
\end{itemize}

Now let $U(\Lambda)\subset V,\,F^s_X,\,F^u_X$ are suspensions under $F_{1,0,1}$ compatible system of neighborhoods and the compatible stable, unstable foliations. Let $U^+=U(\Lambda)\cap T^+_X$, $U^-=U(\Lambda)\cap T^-_X$ and $\mathcal W^s_+=F^s_X\cap T^+_X$, $\mathcal W^u_+=F^u_X\cap T^+_X$, $\mathcal W^s_-=F^s_X\cap T^-_X$, $\mathcal W^u_-=F^u_X\cap T^-_X$. By the construction  $\mathcal L^s_X\subset\mathcal W^s_+$, $\mathcal L^u_X\subset\mathcal W^u_-$. Notice that $(V,X)$ possesses a strongly transverse gluing diffeomorphism $\varphi$ (see Fig. \ref{ph}). For such simple plug $(V,X)$ we do not need to change $X$, it is sufficient to modify $\varphi$.
\begin{figure}[htbp]
\center{\includegraphics[width=0.8\linewidth]{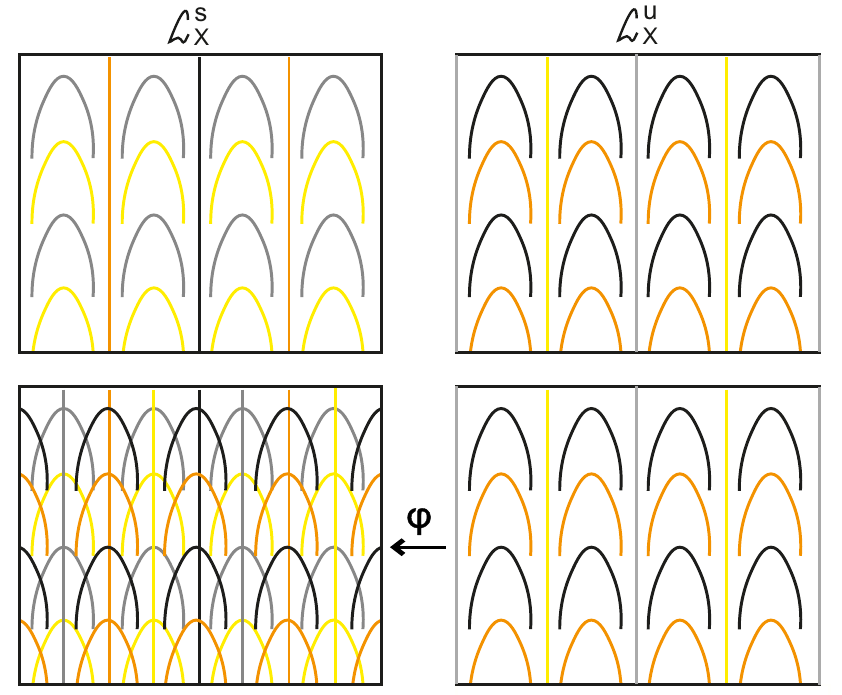}}\caption{Strongly transverse gluing diffeomorphism $\varphi$}
\label{ph}
\end{figure}
\begin{figure}[htbp]
\center{\includegraphics[width=0.8\linewidth]{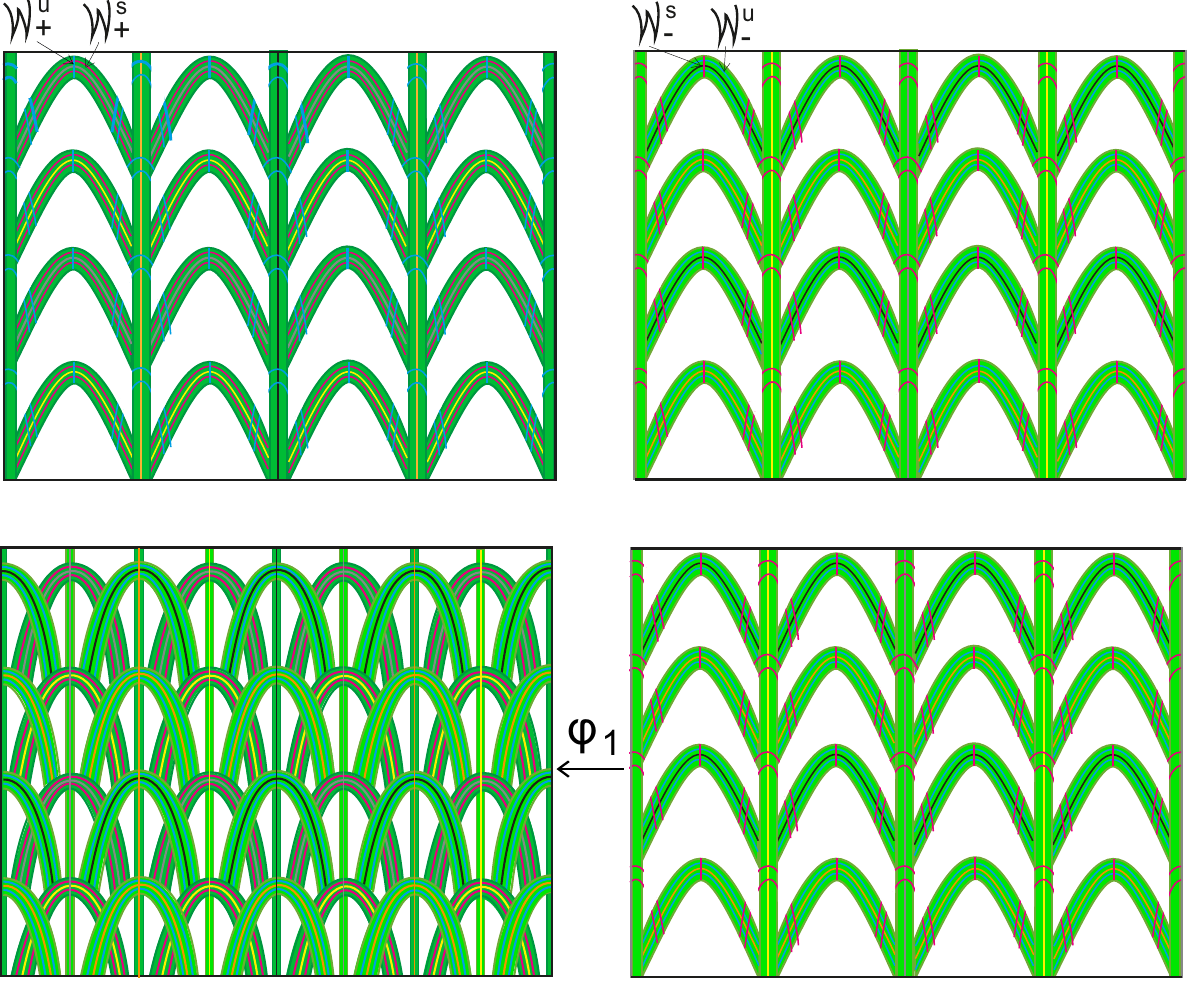}}\caption{Strongly transverse gluing diffeomorphism $\varphi_1$}
\label{ph1}
\end{figure}

We perturb $\varphi$ to $\varphi_1$ such that (see Fig. \ref{ph1}) 
the following properties hold:
\begin{itemize}
\item[1)] the laminations $\varphi_1(\mathcal L^u_X)$ and $\mathcal L^s_X$ are strongly transverse;
\item[2)] the foliations $\varphi_1(\mathcal W^s_-)$ and $\varphi_1(\mathcal W^u_-)$  coincide with $\mathcal W^s_+$ and $\mathcal W^u_+$, respectively,
on $\varphi_1(U^-)\cap U^+$;
\item[3)] $\varphi_1$ is isotopic to $\varphi$ among strongly transverse gluing maps.
\end{itemize}

As the filling lamination $\varphi_1(\mathcal L^u_X)$ is strongly transverse to $\mathcal L^s_X$ then we can extend $\mathcal W^s_+\cup\varphi_1(\mathcal W^s_-),\mathcal W^u_+\cup\varphi_1(\mathcal W^u_-)$ up to a filing MS-foliations $\mathcal G^s_+,\mathcal G^u_+$ on $T^+_X$. Let $\Gamma:T^+_X\setminus\mathcal L^s_X\to T^-_X\setminus\mathcal L^u_X$ be the {\it crossing map}, defined by the flowlines of $X$. Then $\mathcal G^s_-=\mathcal W^s_-\cup\Gamma(\mathcal G^s_+),\mathcal G^u_-=\mathcal W^u_-\cup\Gamma(\mathcal G^u_+)$ are filing MS-foliations  on $T^-_X$. We perturb $\varphi_1$ to $\varphi_2$ such that $\varphi_2(\mathcal G^s_-)=\mathcal G^s_+,\varphi_2(\mathcal G^u_-)=\mathcal G^u_+$ and $\varphi_2|_{U^-}=\varphi_1|_{U^-}$ (see Fig. \ref{ph2}). Moreover, we get $X$-invariant transversal foliations $\mathcal G^s,\mathcal G^u$ on $V$, coinciding with $F^s,F^u$ on $U(\Lambda)$ and such that $\mathcal G^s_\pm=\mathcal G^s\cap T^\pm_X,\mathcal G^u_\pm=\mathcal G^u\cap T^\pm_X$.  
\begin{figure}[htbp]
\center{\includegraphics[width=\linewidth]{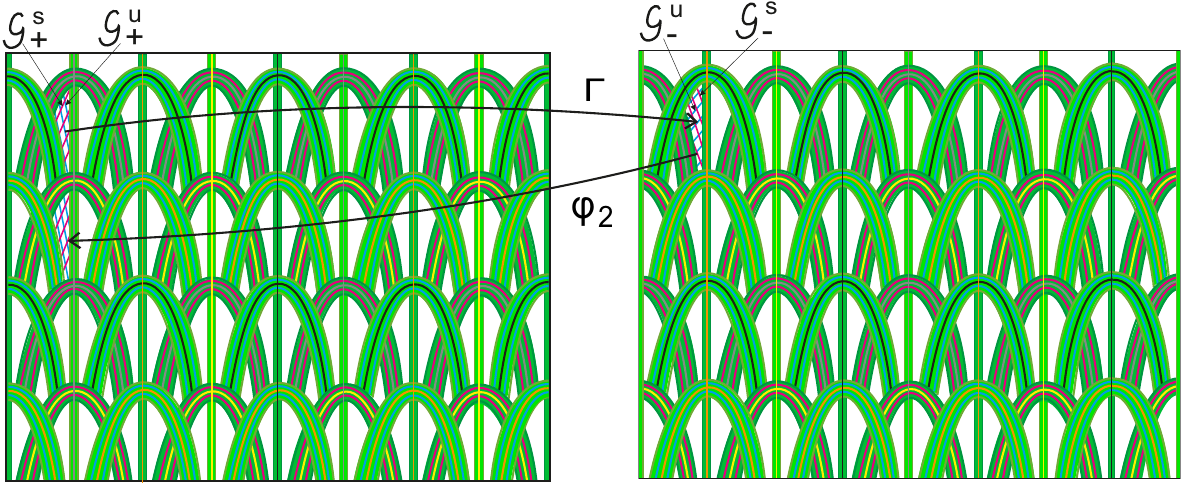}}\caption{Strongly transverse gluing diffeomorphism $\varphi_2$}
\label{ph2}
\end{figure}

Last step to pertube $\varphi_2$ to $\psi$ such that the return map $\theta=\psi\Gamma$ of the orbits of $X$ on $T^+_X$ is hyperbolic. As for a point $x\in T^+_X\setminus\mathcal L^s_X$ 
which is very close to the lamination $\mathcal L^s_X$, its forward orbit will spend a very long time near $\Lambda$, the derivative of $\Gamma$ at $x$ will
expand a vector tangent to $\mathcal G^u$
in by a very large factor. Therefore, 
the derivative of $\psi\Gamma$ expands vectors tangent to $\mathcal G^u$
in by a factor larger than $\lambda$, no matter what $\psi$ might be. Similarly for a point $x\in T^-_X\setminus\mathcal L^u_X$. Then we take $\psi=\psi_+\varphi_2\psi_-$, where $\psi_+:T^+_X\to T^+_X$ coincides with the identity map on $U^+$ and such that $\psi_+\varphi_2\Gamma$ expands vectors tangent to $\mathcal G^u_+$
in by a factor larger than $\lambda$,  $\psi_-:T^-_X\to T^-_X$ coincides with the identity map on $U^-$ and such that $(\Gamma\varphi_2\psi_-)^{-1}$ expands vectors tangent to $\mathcal G^s_-$
in by a factor larger than $\lambda$. 

\section{Countable many graph manifolds, admitting $n$ pairwise inequivalent 
transitive Anosov flows}
In this section we proof Theorem \ref{tigr}. Namely,  for every $n\in\mathbb N$ we construct countable many graph manifolds $M_{k,n}$, $k \in \mathbb{N}$ 
with a single geometric piece, each of which admits $n$ pairwise inequivalent 
transitive Anosov flows.

\begin{proof}
We choose a prime number $p > 2n$ and $q \in \{0, 1, \dots, n\}$. For each 
such pair $(p,q)$, we construct an isotopic to identity gradient-like 
diffeomorphism $F_{p,q} : \mathbb{T}^2 \to \mathbb{T}^2$, its phase 
portrait is shown in Figure \ref{1} for $(p,0)$ and $(p,1)$.
\begin{figure}[htbp]
	\centering
	\includegraphics[width=\linewidth]{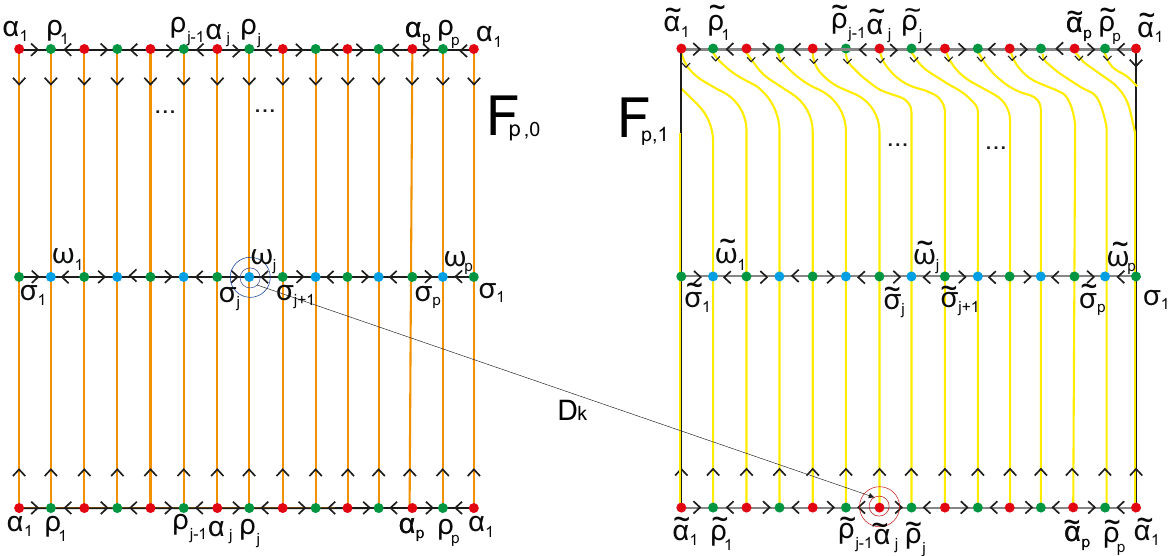}
	\caption{Diffeomorphisms $F_{7,0},\,F_{7,1}$}
	\label{1}
\end{figure}
Diffeomorphism $F_{p,0}$ is the direct product of sink-source map  with $p$-sink-$p$-source map of the circles. So, the non-wandering set of the diffeomorphism $F_{p,0}$ consists of $p$ 
sinks $\omega_1,\dots,\omega_p$, $p$ sources $\alpha_1,\dots,\alpha_p$, and $2p$ saddles $\sigma_1,\dots,\sigma_p,\rho_1,\dots,\rho_p$.

The non-wandering set of the diffeomorphism $F_{p,q},\,q>0$ consists of $p$ 
sinks $\tilde\omega_1,\dots,\tilde\omega_p$, $p$ sources $\tilde\alpha_1,\dots,\tilde\alpha_p$, and $2p$ saddles $\tilde\sigma_1,\dots,\tilde\sigma_p,\tilde\rho_1,\dots,\tilde\rho_p$. Herewith for every $j\in\{1,\dots,p\}$ $${\rm cl}\,W^u_{\tilde\sigma_j}=\tilde\omega_{j-1}\cup W^u_{\tilde\sigma_j}\cup\tilde\omega_j,\,\,{\rm cl}\,W^s_{\tilde\sigma_j}=\tilde\alpha_{j-q}\cup W^u_{\tilde\alpha_j}\cup\tilde\alpha_j,$$ 
$${\rm cl}\,W^u_{\tilde\rho_j}=\tilde\omega_{j}\cup W^u_{\tilde\sigma_j}\cup\tilde\omega_{j+q},\,\,{\rm cl}\,W^s_{\tilde\sigma_j}=\tilde\alpha_{j-1}\cup W^u_{\tilde\alpha_j}\cup\tilde\alpha_j.$$ 
Thus the closures of the invariant 
manifolds of the saddle points form two circles with the homotopy 
type $\langle 1,0\rangle$, as well as two circles with the homotopy 
type $\langle -q,p\rangle$. It implies that the diffeomorphisms $F_{p_1,q_1}$ and $F_{p_2,q_2}$ are 
topologically conjugate if and only if $p_1=p_2$ and $q_1=q_2$.

Let $k\in\mathbb N$ and let $F_{p,q,k}$ be the connected sum of the 
diffeomorphisms $F_{p,0}$ and $F_{p,q}$ obtained by pairwise 
identifying  sinks $\omega_j,\,j\in\{1,\dots,p\}$ of the diffeomorphism $F_{p,0}$ 
with sources $\tilde\alpha_j$ of the diffeomorphism $F_{p,q}$ 
using the $k$-th Dehn twist $D_k$ on the corresponding orbit spaces 
(see Figure \ref{2}).
\begin{figure}[htbp]
	\centering
	\includegraphics[width=\linewidth]{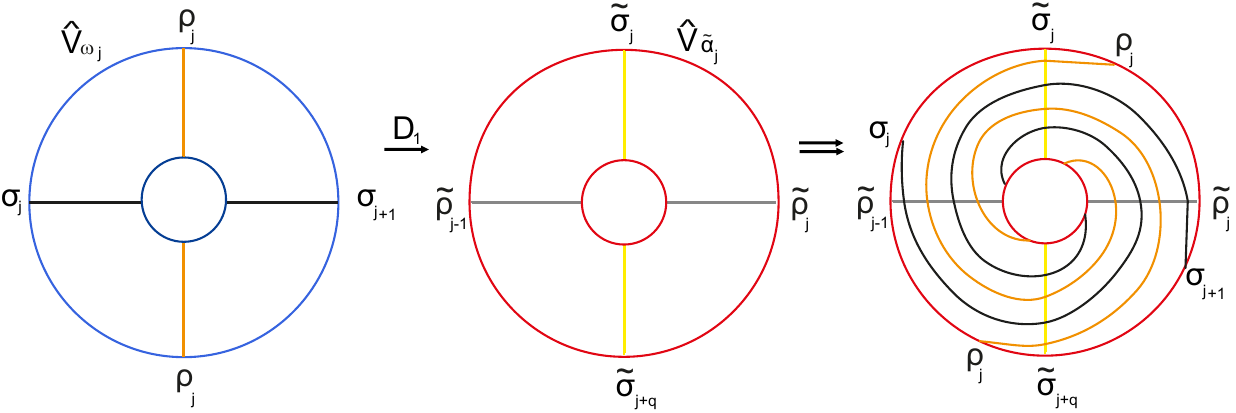}
	\caption{Dehn twist}
	\label{2}
\end{figure}

Thus, the diffeomorphism $F_{p,q,k}$ is a Morse-Smale diffeomorphism on an 
orientable surface of genus $p+1$. Its non-wandering set consists of $p$ 
sinks $\tilde\omega_1,\dots,\tilde\omega_p$, $p$ sources $\alpha_1,\dots,\alpha_p$, and $4p$ saddles $\sigma_1,\dots,\sigma_p,\rho_1,\dots,\rho_p,\tilde\sigma_1,\dots,\tilde\sigma_p,\tilde\rho_1,\dots,\tilde\rho_p$, whose invariant manifolds intersect 
along $4pk$ heteroclinic orbits. The hyperbolic plug $(V,X)$ is obtained from the suspension over the 
diffeomorphism $F_{p,q,k}$ by removing the neighborhoods of the sink 
and source periodic orbits. We note that the topology of $V$ depends 
on $p$ and $k$, but does not depend on $q$. Thus, one can construct $n$ 
pairwise non-equivalent flows on the same plug. The plug $V$ has $p$ 
incoming and $p$ outgoing components. The laminations $\mathcal L^s_X,\,\mathcal L^u_X$ 
on each  component of $T^{+}_X,T^{-}_X$ are homeomorphic the laminations from stable, unstable saddle manifolds $L^s\subset \hat V_{\alpha_j},\,L^u\subset\hat V_{\tilde\omega_j}$, accordingly (see Figure \ref{3}).
\begin{figure}[htbp]
	\centering
	\includegraphics[width=\linewidth]{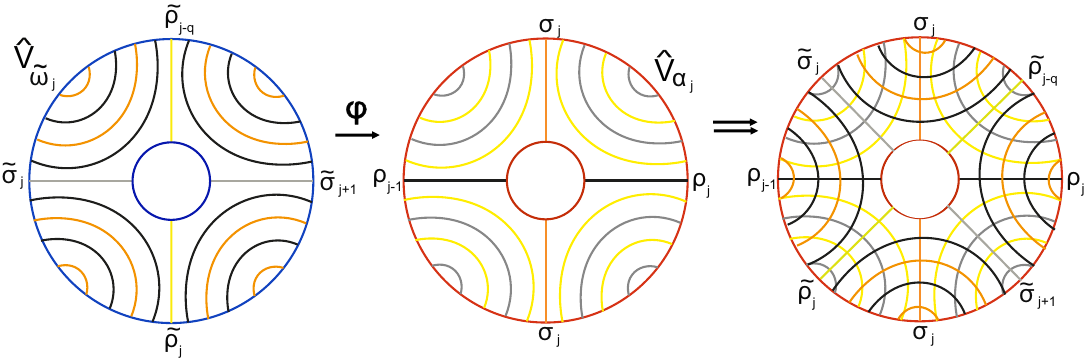}
	\caption{The gluing map $\varphi$}
	\label{3}
\end{figure}

Consequently, the boundaries of the plug can be glued pairwise (preserving 
the indices) by $\varphi$ (see Figure \ref{3}) such that the resulting flow becomes a transitive Anosov flow by Statement \ref{mai}. The obtained manifold $M_{k,n}$ is a graph manifold with a JSJ-decomposition consisting 
of $p$ tori and a single Seifert piece. Let us show that the Anosov flows for different $q$ on the same $M_{k,n}$ are inequivalent. 

Note, that for the manifold $M_{k,n}$ every JSJ-torus $T_j$ arises at the place, where the basins $V_{\omega_j},\,V_{\tilde\alpha_j}$ are glued together. It follows from the Statement \ref{br} that up to a homeomorphism the traces in $T_j$ of the invariant manifolds of the saddle orbits of the set $\Lambda_X$ are an invariant of the topological equivalence of the flow. To understand how the traces on $T_j$ are arranged, we will change the order in which the underlying manifold $M_{k,n}$ was glued. That is, first we glue by $\varphi$ the tori corresponding to the basins $V_{\omega_j},\,V_{\tilde\alpha_j}$, and then the tori corresponding to the basins $V_{\omega_j},\,V_{\tilde\alpha_j}$, using Dehn's rotation. After gluing by $\varphi$, we will see on each torus $\tilde T_j$ a lamination  consisting of 8 knots, 4 of which are traces of unstable manifolds, and the remaining 4 are traces of stable manifolds  of saddle orbits (see Fig. \ref{4}). 
\begin{figure}[htbp]
	\centering
	\includegraphics[width=\linewidth]{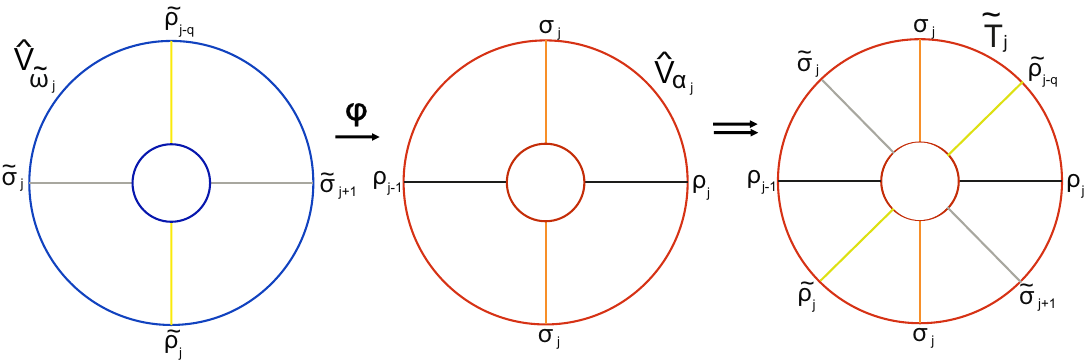}
	\caption{Lamination on  $\tilde T_j$}
	\label{4}
\end{figure}

The traces of the unstable orbits are carried by the flow into the neighborhood  of the attracting periodic orbits, while the traces of the stable orbits are carried into the neighborhood of the repelling ones. Thus, at the boundaries of the attracting and repelling periodic orbits, the traces form laminations, as shown in the Figure \ref{5}, that gives the resultin lamination on $T_j$.
\begin{figure}[htbp]
	\centering
	\includegraphics[width=\linewidth]{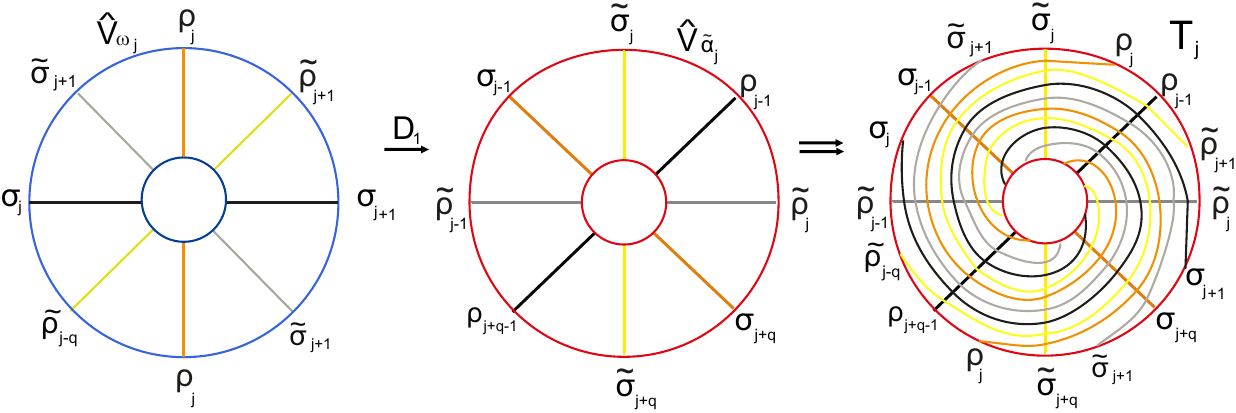}
	\caption{Lamination on $T_j$}
	\label{5}
\end{figure}
The presence on $T_j$ of traces of invariant manifolds of the periodic orbits of saddles $\tilde\sigma_j,\tilde\sigma_{j+1}$ indicates that the cyclic order of the numbers of periodic orbits of saddles $\tilde\sigma_1,\dots,\tilde\sigma_{p}$ is an invariant of the constructed Anosov flow. On the other hand, the presence in $T_j$ of a trace from the invariant manifold of the periodic orbit of $\tilde\sigma_{j+q}$ indicates that the number $q$ is also an invariant of the flow.
\end{proof}


\begin{thebibliography}{99}
\bibitem{BaJ} T. Barbot. Mise en position optimale de tores par rapport a un flot d'Anosov // Comment. Math. Helv. 1995. V. 70. P. 113--160.

\bibitem{Barb2} T. Barbot. Flots d'Anosov sur les vari'et'es graph'ees au sens de Waldhausen // Ann. Inst. Fourier
Grenoble. 1996. V. 46. P. 1451--1517.

\bibitem{BaFe3} T. Barbot, S. Fenley.  Orbital equivalence classes of finite coverings of geodesic flows // 
arXiv:2205.02495.

\bibitem{BBY} F. Beguin, C. Bonatti, B. Yu.  Building Anosov flows on 3-manifolds // Geom. Topol. 2017. V. 21. No. 3. P. 1837--1930.


\bibitem{BGLP} Bonatti C., Grines V., Laudenbach F., Pochinka O. Topological classification of Morse-Smale diffeomorphisms without heteroclinic curves on 3-manifolds // Ergodic Theory and Dynamical Systems. 2019. V. 39. No. 9. P. 2403--2432.

\bibitem{Br} M. Brunella. Separating the basic sets of a non-transitive Anosov flow // Bull. London Math. Soc. 1993. V. 25. No. 5. P. 487--490.

\bibitem{Ca} D. Calegari. Foliations and the geometry of 3-manifolds // Oxford Mathematical Monographs, Oxford
University Press. Oxford. 2007.

\bibitem{CC} A. Candel, L. Conlon.  Foliations I and II // Providence, Rhode Island: American Mathematical Society.  2000.  394 p. (Graduate Studies in Mathematics. V. 23) and  Providence, Rhode Island: American Mathematical Society.  2003. 545 p. (Graduate Studies in Mathematics. V. 60).

\bibitem{ClPi} A. Clay, T. Pinsky.  Graph manifolds that admit arbitrarily many Anosov flows // arXiv:2006.09101

\bibitem{Gh} E. Ghys. Flots d'Anosov sur les 3-vari'et'es fibr'ees en cercles.  Ergodic Theory Dynam. Systems. 1984. V. 4. No. 1. P. 67--80.

\bibitem{Goo} S. Goodman.  Dehn surgery on Anosov flows //  Geometric dynamics (Rio de Janeiro, 1981), Lecture
Notes in Math. 1983. V. 1007, P. 300--307.

\bibitem{KH} A. Katok, B. Hasselblatt, Introduction to the modern theory of dynamical systems, Cambridge
University Press. 1995.

\bibitem{Marg} G. Margulis, Y-flows on three dimensional manifolds // Appendix to Certain smooth ergodic systems. Uspehi Mat. Nauk. 1967. V. 22. No. 5. P. 107--172.

\bibitem{Pa} Palis J. On Morse-Smale dynamical systems //Topology.  1969.  V. 8.  No. 4.  P. 385--404.
	
	\bibitem{PS}  Palis J., Smale S. Structural stability theorems // The Collected Papers of Stephen Smale. 2000. V. 2. – P. 739--747.

\bibitem{Pla81} J. F. Plante. Anosov flows, transversely affine foliations and a conjecture of Verjovsky // J. London Math.
Soc. 1981. V. 23. No. 2. P. 359--362.

\bibitem{PT} J. Plante, W. Thurston. Anosov flows and the fundamental group // Topology. 1972. V. 11. P. 147--150.
\end{thebibliography}
\end{document}